\documentclass[12pt, reqon]{amsart}

\usepackage{cite}

\usepackage{color}
 \usepackage{amscd}

\usepackage{amsrefs}

 \usepackage{amssymb, amsmath}

 \input xy
  \xyoption{all}
 
 \input diagxy

\usepackage{epsfig}

    \newtheorem{theorem}{Theorem}[section]
\newtheorem{prop}{Proposition}[section]
\newtheorem{lemma}[theorem]{Lemma}
\newtheorem{cor}[theorem]{Corollary}
  \numberwithin{equation}{section}
 \numberwithin{prop}{section}
\newtheorem{exmp}{Example}[section]

\newcommand{\Beq}{\begin{equation}}
\newcommand{\Eeq}{\end{equation}}
\newcommand{\Beqr}{\begin{eqnarray}}
\newcommand{\Eeqr}{\end{eqnarray}}

\newcommand{\rmc}{{\mathbf {Cat}}}

\newcommand{{\wtlg}}{\widetilde\gamma }

\newcommand{{\wtlG}}{\widetilde\Gamma }

\newcommand{{\wtlv}}{\widetilde v }

\newcommand{{\tlg}}{\tilde\gamma }

\newcommand{{\tlG}}{\tilde\Gamma }

 \newcommand{\mbl}{{{\mathbf \Theta}}}

 \newcommand{\mbeta}{{{\mathbf \Psi}}}

\newcommand{\Obj}{{\rm Obj }}
\newcommand{\Mor}{{\rm Mor }}

\begin{document}
\title{Categorified presheaves and sieves}

\author{Saikat Chatterjee }
\address{School of Mathematics, Indian Institute of Science Education and Research\\
CET Campus\\ Thiruvananthapuram, Kerala-695016\\
India}
\email{saikat.chat01@gmail.com}



\def\xypic{\hbox{\rm\Xy-pic}}

\begin{abstract}
 Let $\mathcal C$ be a category of  a set of  (small) categories.  This paper concerns with the ${\mathbf {Cat}}$-valued presheaves and sieves  over category $\mathcal C.$ Since ${\mathbf {Cat}}$ is not a concrete category, existing definition of presheaves can not deal with the situation. This paper proposes a new framework for the purpose.   The main result is a version of  Yoneda embedding for  ${\mathbf  {Cat}}$-valued presheaves, and construction of the ${\mathbf {Cat}}$-valued sieves over the category ${\mathcal O}(\mathbf B)$ of subcategories of a given topological category $\mathbf B.$

 \end{abstract}



\maketitle

\section{Introduction}\label{s:int}
The objective of this paper is to propose a framework for $\mathbf {Cat}$-\textit{valued presheaves and sieves}. This paper can be considered as the first part of a sequel. The second part [13] contains the discussions on ${\mathbf {Cat}}$-\textit{valued sheaves}. On one hand, this paper provides the necessary mathematical accessories required for the second part.  On the other, it lays the foundation for a paper  under preparation (jointly with  A Lahiri, and A N Sengupta), which extends the construction of sieves in this paper to \textit{Grothendieck topologies} on a ``\textit{topological category}''.

Recall a   presheaf, for a given category $\mathbf C,$  is a contravariant functor [6,15,17]
\begin{equation}\label{E:introtradpre}
	R:{\mathbf C}^{\rm op}\longrightarrow {\mathbf {Set}},
\end{equation}
where ${\mathbf {Set}}$ is the (locally small) category of small sets. ${\mathbf C}$ is often chosen to be the category ${\widetilde {\mathcal O}}(B)$ of open subsets of a topological space $B$, that is, 
\begin{equation}\label{E:introtildeo}
\begin{split}
&\Obj\biggl({\widetilde {\mathcal O}}(B)\biggr) := \{U|\hskip 0.15 cm U\subset B\},\\
&{\rm Hom}(U, V) := \{f:U\to V|\hskip 0.15 cm U, V\subset B\}.
\end{split}
\end{equation}
In \eqref{E:introtradpre} instead of ${\mathbf {Set}},$ we may take  any other \textit{concrete category}. For instance by taking category of (small) groups ${\mathbf {Grp}}$, category of (small) vector spaces ${\mathbf {Vect}}$ or category of (small) rings ${\mathbf {Ring}}$ as the codomain in  \eqref{E:introtradpre}, and   we respectively get \textit{presheaf of groups, presheaf of vector spaces} or \textit{presheaf of rings}.

Now suppose instead of a topological space $B,$ we are concerned with a category $\mathbf B,$ whose both object and morphism spaces are topological spaces, namely $\mathbf B$ is a  \textit{topological category}. Note that this taxonomy of topological category is not standard in literature. Ours is consistent with [7]. Natural object of interest in that case would be the category ${\widetilde {\mathcal O}}(\mathbf B)$ of (``open") subcategories  \footnote{In this paper we will not discuss \textit{open subcategories}. In [13] it has been defined. It would serve the purpose of this paper, if we simply  think of ${\widetilde {\mathcal O}}(\mathbf B)$ as a category of subcategories of a given category $\mathbf B.$}  of $\mathbf B.$ The question  we may ask is, then what should be a ``natural framework  for presheaves" in this context. Of course, one may still work with the definition in  \eqref{E:introtradpre}. But, more natural choice would be to consider the category ${\mathbf {Cat}}$ of small categories as the codomain  of a presheaf in this context, rather than ${\mathbf {Set}}$ (or any other concrete category). Since ${\mathbf {Cat}}$ is not a concrete category, we can not proceed with the definition of presheaf given in \eqref{E:introtradpre}  and we need a new framework. In this paper we propose such a framework, and develop  the corresponding theory of ${\mathbf {Cat}}$-\textit{valued sieves} 

Let $\mathcal C$ be a category of a collection of small categories. We work with a ${\mathbf {Cat}}$-\textit{valued presheaf} over $\mathcal C$ given by a contravariant functor:
\begin{equation}\label{E:introtradprecat}
{\mathcal R}: {\mathcal C}^{\rm op}\longrightarrow {\mathbf {Cat}}.
\end{equation}
In particular, our future goal is to consider the case ${\mathcal C}={\widetilde {\mathcal O}}({\mathbb P}M),$ and study the Grothendieck topologies [1, 2, 20] on the \textit{path space groupoid}  ${\mathbb P}M$ of a given smooth manifold $M$; that is, a category  ${\mathbb P}M$, whose object space is the manifold $M$ and morphisms are certain equivalence classes of smooth paths [5, 8, 9, 11, 24]. Usual compact-open topology defines a topology on $\Mor({\mathbb P}M).$ The path space groupoid over a smooth manifold plays a pivotal role in higher gauge theories [4, 16, 21--23]. The ``locally defined subcategories" of ${\mathbb P}M$ also appeared in the context of local structures of categorical principal bundles [8, 12]. 
 
 \subsection*{Notation}
 We work with following set of notation. For $\mathbf C$ and $\mathbf D,$  a given pair of categories.
\begin{equation}\label{E:notfunc}
	{\rm Fun}(\mathbf C, \mathbf D)
\end{equation}
will denote the set of all functors from $\mathbf C$ to $\mathbf D.$ For functors  from  $\mathbf C$ to $\mathbf D,$       
\begin{equation}\label{E:notnat}
	{\mathcal N}(\mathbf C, \mathbf D)
\end{equation}
will denote the set of all natural transformations.

If ${\mathbf \theta}_1, {\mathbf \theta}_2:\mathbf C\to \mathbf D$ are a pair of functors, then 
\begin{equation}\label{E:notnatspfunc}
	{\rm Nat}(\mathbf \theta_1, \theta_2)
\end{equation}
is the set of  natural transformations between $\theta_1$ and $\theta_2.$  We denote a natural transformation $\Phi$ from a functor $\theta_1$ to another functor $\theta_2$ as
\begin{equation}
\Phi:\theta_1\Longrightarrow \theta_2.
\end{equation}
Let
\begin{equation}\label{E:notfunccat}
	{\mathcal F}(\mathbf C, \mathbf D);
\end{equation}
be the category of functors;
that is, 
\begin{equation}\label{E:notfunccatobjmor}
	\begin{split}
		&\Obj\biggl({\mathcal F}(\mathbf C, \mathbf D)\biggr)={\rm Fun}(\mathbf C, \mathbf D)\\
		&\Mor\biggl({\mathcal F}(\mathbf C, \mathbf D)\biggr)={\mathcal N}(\mathbf C, \mathbf D).
\end{split}
\end{equation}

Given a morphism $f$ in some category, $s(f), t(f)$ will respectively denote the source of $f$ and target of $f;$ that is,
$${s(f)}\xrightarrow{f}{t(f)}.$$

$\mathbf {\emptyset}$ will be  the \textit{empty category}; i.e. a category whose object and morphism sets are empty sets.

  \subsection*{Summary of the paper}    
  
We start with a category of a collection of small categories ${\mathcal C}.$  We show that there exists a (contravariant) functor ${\mathcal F}_{\mathbf U}:{\mathcal C}^{\rm op}\longrightarrow {\mathbf {Cat}}$ corresponding to each ${\mathbf U}\in \Obj(\mathcal C),$ analogous to a ${\rm Hom}$ functor in set theoretic set-up. We show, in Proposition~\ref{Pr:yoneda}, that functors ${\mathcal F}_{\mathbf U}$  ``partially" fulfill Yoneda lemma.

We define a $\mathbf {Cat}$-valued presheaf on ${\mathcal C}$ to be a contravariant functor
$${\mathcal C}^{\rm op}\longrightarrow {\mathbf {Cat}}.$$
The main result of this paper is  Theorem~\ref{Th:yonembed}. In Theorem~\ref{Th:yonembed} we show that classical Yoneda embedding is still valid in this framework; that is, we can realize any category $\mathcal C$ as above, as a full subcategory of category of its  ${\mathbf {Cat}}$-valued presheaves. We move onto define $\mathbf {Cat}$-valued sieves over ${\mathcal C}.$ In Example~\ref{Ex:sie} we construct ${\mathbf {Cat}}$-valued sieves over the category ${\mathcal O}(\mathbf B)$ (defined in \eqref{E:ob}).

\section{Functors to $\mathbf {Cat}$}\label{S:mbsieve}

Let $\mathcal C$ be a category of a collection of (small)categories; that is objects are a set of (small)categories and morphisms are functors between them. In particular, we will be interested in category ${\widetilde {\mathcal O}}(\mathbf B),$ where $\mathbf B$ is a given category and, 
\begin{equation}\label{E:genob}
	\begin{split}
		&{\rm Obj}\biggl({\widetilde {\mathcal O}}({\mathbf B})\biggr):=\{{\mathbf U}| {\mathbf U}\subset {\mathbf B}\}=\hbox{set of all subcategories of}\hskip 0.2 cm {\mathbf B},\\
	 &{\rm Hom}({\mathbf U}, {\mathbf V})=\{{\mbl}:{\mathbf U}\longrightarrow {\mathbf V}|{\mathbf U}, {\mathbf V}\subset {\mathbf B}\}.	
\end{split}
\end{equation}

We will also work with the category ${{\mathcal O}}(\mathbf B),$ whose objects are same as those of ${\widetilde {\mathcal O}}(\mathbf B);$ but, only morphism  between any two subcategories (objects) is the  inclusion functor, if one is  subcategory  of the other. Otherwise no morphism exists:

\begin{equation}\label{E:ob}
	\begin{split}
		&{\rm Obj}\biggl({\mathcal O}({\mathbf B})\biggr):=\{{\mathbf U}| {\mathbf U}\subset {\mathbf B}\}=\hbox{set of all subcategories of}\hskip 0.2 cm {\mathbf B},\\
	 &{\rm Hom}({\mathbf U}, {\mathbf V})=\{{\mathbf i}:{\mathbf U}\hookrightarrow {\mathbf V}|{\mathbf U}\subset {\mathbf V}\subset {\mathbf B}\},\\ 
		&\hskip 02.4 cm \hbox{where}\hskip 0.2cm {\mathbf i}\hskip 0.2 cm \hbox{is the inclusion functor, and}\\
	 &{\rm Hom}({\mathbf U}, {\mathbf V})=\emptyset, \hbox {if}\hskip 0.2 cm {\mathbf U}\not\subset {\mathbf V}.
\end{split}
\end{equation}

Let $\rmc$ be the category of all (small) categories. We define the following contravariant  functor, which will play the role of $\rm Hom$-functor in set theoretic framework. We define  ${\mathcal F}_{\mathbf U}:{\mathcal C}^{\rm op}\to \rmc,$ corresponding to each $\mathbf U\in {\rm Obj}(\mathcal C),$ to be
\begin{eqnarray}
	&{\mathcal F}_{\mathbf U}:&{\rm Obj}(\mathcal C)\to {\rm Obj}(\rmc)\nonumber\\
	       && \mathbf V\mapsto {\mathcal F}(\mathbf V, \mathbf U)\label{E:objfun}\\
	&{\mathcal F}_{\mathbf U}:&{\rm Mor}(\mathcal C)\to {\rm Mor}(\rmc)\nonumber\\
 &&\biggl({\mathbf V}\xrightarrow{\mbl} {\mathbf W}\biggr)\mapsto \biggl({\mathcal F}(\mathbf W, \mathbf U)\xrightarrow{{\mathcal F}_{\mathbf U}(\mbl)}{\mathcal F}(\mathbf V, \mathbf U)\biggr),\label{E:morfun}
\end{eqnarray}

where $\mathbf \Theta$ is a functor from the category $\mathbf V$ to $\mathbf W.$ \eqref{E:morfun} definitely requires an explanation. Let us verify that indeed we have such a functor. Since $\rmc$ is the category of categories, \eqref{E:objfun} does make sense. Now suppose $\mathbf V, \mathbf W\in {\rm Obj}(\mathcal C),$ and we are given a functor $$\mbl:\mathbf V\to \mathbf W.$$
Under the action of ${\mathcal F}_{\mathbf U}$, $\mathbf V, \mathbf W$ respectively mapped to the categories ${\mathcal F}(\mathbf V, \mathbf U)$ and ${\mathcal F}(\mathbf W, \mathbf U).$ We have to show that $\mbl$ defines a functor
\begin{eqnarray}
	&{\mathcal F}_{\mathbf U}(\mbl)&:{\mathcal F}(\mathbf W, \mathbf U)\longrightarrow {\mathcal F}(\mathbf V, \mathbf U),\nonumber\\
 &{\mathcal F}_{\mathbf U}(\mbl):&\Obj\biggl({\mathcal F}(\mathbf W, \mathbf U)\biggr)\longrightarrow \Obj\biggl({\mathcal F}(\mathbf V, \mathbf U)\biggr),\nonumber\\
&&{\rm Fun}(\mathbf W, \mathbf U)\longrightarrow {\rm Fun}(\mathbf V, \mathbf U)\label{E:funcfuncobj}\\
&{\mathcal F}_{\mathbf U}(\mbl):&\Mor\biggl({\mathcal F}(\mathbf W, \mathbf U)\biggr)\longrightarrow \Mor\biggl({\mathcal F}(\mathbf V, \mathbf U)\biggr),\nonumber\\
&&{\mathcal N}(\mathbf W, \mathbf U)\longrightarrow {\mathcal N}(\mathbf V, \mathbf U)\label{E:funcfuncmor}.
\end{eqnarray}
For any $\mbeta\in \Obj\biggl({\mathcal F}(\mathbf W, \mathbf U)\biggr)={\rm Fun}(\mathbf W, \mathbf U),$ by composition with the functor $\mbl:\mathbf V\to \mathbf W,$ we get a functor ${\mathbf \Psi}\mbl:\mathbf V\to \mathbf U;$ that is ${\mathbf \Psi} \mbl\in {\rm Fun}(\mathbf V, \mathbf U)=\Obj\biggl({\mathcal F}(\mathbf V, \mathbf U)\biggr).$ So, we have the map 
\begin{equation}\label{D:objfunc}
	\begin{split}
		{\mathcal F}_{\mathbf U}(\mbl):\Obj\biggl({\mathcal F}(\mathbf W, \mathbf U)\biggr)&\longrightarrow \Obj\biggl({\mathcal F}(\mathbf V, \mathbf U)\biggr),\\
		 \mbeta&\mapsto \mbeta\mbl,\\
	& \xymatrixcolsep{5pc}
\xymatrix{
	& {\mathbf V} \ar[d]^{\mbeta\mbl} \ar@{-->}[dl]_-{\mbl} \\
	{\mathbf W} \ar[r]^{\mbeta} &{\mathbf U}
}
\end{split}
\end{equation}
Now suppose ${\mathcal S}:\mbeta_1\Longrightarrow \mbeta_2$ is a natural transformation between the functors $\mbeta_1, \mbeta_2:\mathbf W\to \mathbf U.$ That is for any 
$a\xrightarrow{f}b\in {\rm Mor}(\mathbf W),$ following diagram commutes in $\mathbf U:$

\begin{equation} \label{D:yonfun}
\xymatrix{
         \ar[d]^{{\mathcal S}(a)} \mbeta_1(a)     \ar[rr]^-{\mbeta_1(f)} & &\mbeta_1(b) \ar[d]_{{\mathcal S}(b)} \\
\mbeta_2(a) \ar[rr]_-{\mbeta_2(f)}& & \mbeta_2(b) 
}.
\end{equation}
Now considering  above commutative diagram in the subcategory ${\rm Image}(\mbl)\subset \mathbf W,$ we have, for any ${\tilde a}\xrightarrow{\tilde f}{\tilde b}\in {\rm Mor}(\mathbf V)$  
\begin{equation} \label{D:yonfunmap}
\xymatrix{
	\ar[d]^{{\mathcal S}(\mbl(\tilde a))} \mbeta_1(\mbl(\tilde a))     \ar[rr]^-{\mbeta_1(\mbl(\tilde f))} & &\mbeta_1(\mbl(\tilde b)) \ar[d]_{{\mathcal S}(\mbl(\tilde b))} \\
	\mbeta_2(\mbl (\tilde a)) \ar[rr]_-{\mbeta_2(\mbl (\tilde f))}& & \mbeta_2(\mbl(\tilde b)) 
}
\end{equation}
commutes. But according to \eqref{D:objfunc}, $${\mathcal F}_{\mathbf U}(\mbl)(\mbeta_i)=\mbeta_i\mbl\in {\rm Fun}(\mathbf V, \mathbf U), \qquad i=1, 2.$$
Thus  commutative diagram \eqref{D:yonfunmap} implies ${\mathcal S}\mbl$ is a natural transformation between ${\mathcal F}_{\mathbf U}(\mbl)(\mbeta_1)$ and ${\mathcal F}_{\mathbf U}(\mbl)(\mbeta_2),$
$${\mathcal S}\mbl:{\mathcal F}_{\mathbf U}(\mbl)(\mbeta_1)\Longrightarrow {\mathcal F}_{\mathbf U}(\mbl)(\mbeta_2).$$
We define
\begin{equation}\label{E:defnat}
	\begin{split}
	&{\mathcal F}_{\mathbf U}(\mbl):\Mor\biggl({\mathcal F}(\mathbf W, \mathbf U)\biggr)\longrightarrow \Mor\biggl({\mathcal F}(\mathbf V, \mathbf U)\biggr),\\
&{\mathcal F}_{\mathbf U}(\mbl):{\mathcal N}(\mathbf W, \mathbf U)\longrightarrow {\mathcal N}(\mathbf V, \mathbf U),\\
&\bigl({\mathcal F}_{\mathbf U}(\mbl)\bigr)(\mathcal S):={\mathcal S}\mbl\in {\mathcal N}(\mathbf V, \mathbf U).
	\end{split}
\end{equation}
It follows from the composition of two natural transformations [19]:
$$\bigl({\mathcal S}_2\circ{\mathcal S}_1\bigr)(a)={\mathcal S}_2(a)\circ {\mathcal S}_1(a), \qquad \forall a\in \Obj(\mathbf W),$$
that ${\mathcal F}_{\mathbf U}(\mbl)$ is a functor.
Thus, using \eqref{D:objfunc} and \eqref{E:defnat}, we obtain the map:
\begin{equation}
	{\mathcal F}_{\mathbf U}:{\rm Mor}(\mathcal C)\to {\rm Mor}(\rmc).\label{E:mormbymbf}
 \end{equation}
 Combining \eqref{E:objfun} and \eqref{E:mormbymbf} we produce our desired maps:
 \begin{equation}\label{E:morobjmbymbf}
	 \begin{split}
		&{\mathcal F}_{\mathbf U}:{\rm Obj}(\mathcal C)\to {\rm Obj}(\rmc),\\
		 &{\mathcal F}_{\mathbf U}:{\rm Mor}(\mathcal C)\to {\rm Mor}(\rmc).
	 \end{split}
 \end{equation}
It is also obvious that ${\mathcal F}_{\mathbf U}$ is functorial. In summary, we have shown that
\begin{prop}\label{Pr:Hompresheaf}
	Let $\mathcal C$ be a category of a collection of (small)categories and $\rmc$ be the category of all (small) categories. Then, for each ${\mathbf U}\in \Obj(\mathcal C),$ we have a contravariant functor ${\mathcal F}_{\mathbf U}:{\mathcal C}^{\rm op}\to \rmc$ given by
	\begin{eqnarray}
	&{\mathcal F}_{\mathbf U}:&{\rm Obj}(\mathcal C)\to {\rm Obj}(\rmc)\nonumber\\
	       && \mathbf V\mapsto {\mathcal F}(\mathbf V, \mathbf U)\nonumber\\
	&{\mathcal F}_{\mathbf U}:&{\rm Mor}(\mathcal C)\to {\rm Mor}(\rmc)\nonumber\\
 &&\biggl({\mathbf V}\xrightarrow{\mbl} {\mathbf W}\biggr)\mapsto \biggl({\mathcal F}(\mathbf W, \mathbf U)\xrightarrow{{\mathcal F}_{\mathbf U}(\mbl)}{\mathcal F}(\mathbf V, \mathbf U)\biggr),\nonumber
	\end{eqnarray}
	where ${\mathcal F}_{\mathbf U}(\mbl)$ is as given in \eqref{D:objfunc} and \eqref{E:defnat}.
\end{prop}
Similarly we can also define the covariant counterpart  ${\overline {{\mathcal F}}_{\mathbf U}}$ of ${\mathcal F}_{\mathbf U}$ as
\begin{eqnarray}
	&{\overline {\mathcal F}}_{\mathbf U}:&{\mathcal C}\to \rmc,\nonumber\\
	&{\overline {\mathcal F}}_{\mathbf U}:&{\rm Obj}(\mathcal C)\to {\Obj}(\rmc)\nonumber\\
	&& \mathbf V\mapsto {\mathcal F}(\mathbf U, \mathbf V)\label{covobj}\\
	&{\overline {\mathcal F}}_{\mathbf U}:&{\rm Mor}(\mathcal C)\to {\rm Mor}(\rmc)\nonumber\\
	&&\biggl({\mathbf V}\xrightarrow{\mbl} {\mathbf W}\biggr)\mapsto \biggl({\mathcal F}(\mathbf U, \mathbf V)\xrightarrow{{\overline {\mathcal F}}_{\mathbf U}(\mbl)}{\mathcal F}(\mathbf U, \mathbf W)\biggr),\label{covmor}
\end{eqnarray}
where the right hand side of \eqref{covmor} is given as follows.
\begin{equation}\label{D:objfunccov}
	\begin{split}
		{\overline {\mathcal F}}_{\mathbf U}(\mbl):\Obj\biggl({\mathcal F}(\mathbf U, \mathbf V)\biggr)&\longrightarrow \Obj\biggl({\mathcal F}(\mathbf U,  \mathbf W)\biggr),\\
		 \mbeta&\mapsto \mbl\mbeta,\\
	& \xymatrixcolsep{5pc}
\xymatrix{
	& {\mathbf U} \ar[d]^{\mbl\mbeta} \ar@{-->}[dl]_-{\mbeta} \\
	{\mathbf V} \ar[r]^{\mbl} &{\mathbf W}
}
\end{split}
\end{equation}
and, 
\begin{equation}\label{E:defnatcov}
	\begin{split}
	&{\overline {\mathcal F}}_{\mathbf U}(\mbl):\Mor\biggl({\mathcal F}(\mathbf U, \mathbf V)\biggr)\longrightarrow \Mor\biggl({\mathcal F}(\mathbf  V, \mathbf W)\biggr),\\
&{\overline {\mathcal F}}_{\mathbf U}(\mbl):{\mathcal N}(\mathbf U, \mathbf V)\longrightarrow {\mathcal N}(\mathbf U, \mathbf W),\\
&\bigl({\overline {\mathcal F}}_{\mathbf U}(\mbl)\bigr)(\mathcal S):=\mbl{\mathcal S}\in {\mathcal N}(\mathbf U, \mathbf W).
	\end{split}
\end{equation}

Let $\mathbf U, \mathbf V\in \Obj(\mathcal C).$ Then, by Proposition~\ref{Pr:Hompresheaf}, we have a pair of functors ${\mathcal F}_{\mathbf U}, {\mathcal F}_{\mathbf V}:{\mathcal C}^{\rm op}\to \rmc.$ On the other hand, given any $\mbl:\mathbf U\to \mathbf V,$ according to \eqref{covmor}, we have the functor

\begin{equation}\nonumber
{{\overline {\mathcal F}}_{\mathbf W}(\mbl)}:{\mathcal F}(\mathbf W, \mathbf U)\longrightarrow {\mathcal F}(\mathbf W, \mathbf V).
\end{equation}
Note that, by definition [see \eqref{E:objfun}] , ${\mathcal F}(\mathbf W, \mathbf U)={\mathcal F}_{\mathbf U}(\mathbf W)$ and ${\mathcal F}(\mathbf W, \mathbf V)={\mathcal F}_{\mathbf V}(\mathbf W).$ Hence,
\begin{equation}\label{convfunnat}
{{\overline {\mathcal F}}_{\mathbf W}(\mbl)}:{\mathcal F}_{\mathbf U}(\mathbf W)\longrightarrow {\mathcal F}_{\mathbf V}(\mathbf W).
\end{equation}
In fact ${{\overline {\mathcal F}}(\mbl)},$ defined as 
\begin{equation}\label{E:natutransmap}
	\biggl({{\overline {\mathcal F}}(\mbl)}\biggr)(\mathbf W):={{\overline {\mathcal F}}_{\mathbf W}(\mbl)},
\end{equation}	
	is a natural transformation from ${\mathcal F}_{\mathbf U}$ to ${\mathcal F}_{\mathbf V}:$
$${{\overline {\mathcal F}}(\mbl)}:{\mathcal F}_{\mathbf U}\Longrightarrow{\mathcal F}_{\mathbf V};$$
that is, following diagram commutes for all $\mathbf W, \mathbf W'\in \Obj(\mathcal C)$ and ${\mathbf \Phi}\in{\rm Hom}(\mathbf W', \mathbf W)={\rm Fun}(\mathbf W', \mathbf W):$
\begin{equation} \label{D:yoniso}
\xymatrix{
	\ar[d]^{{\overline {\mathcal F}}_{\mathbf W}(\mbl)} {\mathcal F}_{\mathbf U}(\mathbf W)     \ar[rr]^-{{\mathcal F}_{\mathbf U}(\mathbf \Phi)} & & {\mathcal F}_{\mathbf U}(\mathbf W')\ar[d]_{{\overline {\mathcal F}}_{\mathbf W'}(\mbl)} \\
{\mathcal F}_{\mathbf V}(\mathbf W) \ar[rr]_-{{\mathcal F}_{\mathbf V}(\mathbf \Phi)}& & {\mathcal F}_{\mathbf V}(\mathbf W') 
}.
\end{equation}
Let us verify  commutivity of the above diagram.
\begin{lemma}\label{L:commu}
	Diagram in \eqref{D:yoniso} commutes for all ${\mathbf \Phi}\in{\rm Hom}(\mathbf W', \mathbf W)={\rm Fun}(\mathbf W', \mathbf W).$ That means, we have a natural transformation
	$${{\overline {\mathcal F}}(\mbl)}:{\mathcal F}_{\mathbf U}\Longrightarrow{\mathcal F}_{\mathbf V},$$  
	given by \eqref{E:natutransmap}.	
	\end{lemma}
\begin{proof}
	We have to show that diagram commutes both at the level of objects and morphisms. 
	
	First let us verify the commutivity for the objects.

	Let ${\mathbf \Psi}\in \Obj({\mathcal F}_{\mathbf U}(\mathbf W))= {\rm Fun}(\mathbf W, \mathbf U).$ Then, 
	\begin{equation}\label{E:comobj1}
		\begin{split}
			&\biggl({\mathcal F}_{\mathbf U}(\mathbf \Phi)\biggr)({\mathbf \Psi})={\mathbf \Psi}{\mathbf \Phi}\in {\rm Fun}(\mathbf W', \mathbf U)=\Obj({\mathcal F}_{\mathbf U}(\mathbf W')), [\hbox{using} \eqref{D:objfunc}] \\
			&\Rightarrow  \biggl({\overline {\mathcal F}}_{\mathbf W'}\biggr)({\mathbf \Psi}{\mathbf \Phi})=\mbl{\mathbf \Psi}{\mathbf \Phi}\in 	{\rm Fun}(\mathbf W', \mathbf V)=\Obj({\mathcal F}_{\mathbf V}(\mathbf W')) [\hbox{using} \eqref{D:objfunccov}].   
\end{split}
	\end{equation}
	
	On the other hand,
	\begin{equation}\label{E:comobj2}
		\begin{split}
		&\biggl({\overline {\mathcal F}}_{\mathbf W}\biggr)(\mathbf \Psi)=\mbl{\mathbf \Psi}\in 	{\rm Fun}(\mathbf W,  \mathbf V)=\Obj({\mathcal F}_{\mathbf V}(\mathbf W)), \\			
			&\Rightarrow \biggl({\mathcal F}_{\mathbf V}(\mathbf \Phi)\biggr)(\mbl{\mathbf \Psi})=\mbl{\mathbf \Psi}{\mathbf \Phi}\in {\rm Fun}(\mathbf W', \mathbf V)=\Obj({\mathcal F}_{\mathbf V}(\mathbf W')).		
	\end{split}
	\end{equation}	
	 \eqref{E:comobj1}, \eqref{E:comobj2} imply  that the diagram in \eqref{D:yoniso} commutes at the level of objects.
	 
	 Now let us verify commutivity for the morphisms.

	Let ${\mathcal S}\in \Mor({\mathcal F}_{\mathbf U}(\mathbf W))= {\mathcal N}(\mathbf W, \mathbf U).$ Then, 
	\begin{equation}\label{E:commor1}
		\begin{split}
			&\biggl({\mathcal F}_{\mathbf U}(\mathbf \Phi)\biggr)({\mathcal S})={\mathcal S}{\mathbf \Phi}\in {\mathcal N}(\mathbf W', \mathbf U)=\Mor({\mathcal F}_{\mathbf U}(\mathbf W')), [\hbox{using} \eqref{E:defnat}]	\\
			&\Rightarrow  \biggl({\overline {\mathcal F}}_{\mathbf W'}\biggr)({\mathcal S}{\mathbf \Phi})=\mbl{\mathcal S}{\mathbf \Phi}\in 	{\mathcal N}(\mathbf W', \mathbf V)=\Mor({\mathcal F}_{\mathbf V}(\mathbf W'))	[\hbox{using} \eqref{E:defnatcov}]. 
\end{split}
	\end{equation}
	
	On the other hand,
	\begin{equation}\label{E:commor2}
		\begin{split}
		&\biggl({\overline {\mathcal F}}_{\mathbf W}\biggr)(\mathcal S)=\mbl{\mathcal S}\in 	{\mathcal N}(\mathbf W,  \mathbf V)=\Mor({\mathcal F}_{\mathbf V}(\mathbf W)),\\			
			&\Rightarrow \biggl({\mathcal F}_{\mathbf V}(\mathbf \Phi)\biggr)(\mbl{\mathcal S})=\mbl{\mathcal S}{\mathbf \Phi}\in {\mathcal N}(\mathbf W', \mathbf V)=\Mor({\mathcal F}_{\mathbf V}(\mathbf W')),			
	\end{split}
	\end{equation}	
	 \eqref{E:commor1}, \eqref{E:commor2} imply  that the diagram in \eqref{D:yoniso} commutes also at the level of morphisms.

	\end{proof}
	In other words, we have a set map:
	\begin{equation}\label{E:yonmap}
	\begin{split}
	{\overline {\mathcal F}}: &{\rm Fun}(\mathbf U, \mathbf V)\to {\rm Nat}({\mathcal F}_{\mathbf U}, {\mathcal F}_{\mathbf V})\\
	&\mbl\mapsto {\overline {\mathcal F}}(\mbl).
	\end{split}
		\end{equation}
		We propose  a version  of \textit{Yoneda lemma} in this context.
\begin{prop}\label{Pr:yoneda}
There exists an isomorphism between  ${\rm Fun}(\mathbf U, \mathbf V)$ and ${\rm Nat}({\mathcal F}_{\mathbf U}, {\mathcal F}_{\mathbf V}):$
\begin{equation}\label{E:yonlemma}
{\rm Fun}(\mathbf U, \mathbf V)\cong{\rm Nat}({\mathcal F}_{\mathbf U}, {\mathcal F}_{\mathbf V}).
\end{equation}
The set map ${\overline {\mathcal F}}$ in \eqref{E:yonmap} defines the corresponding bijection.
\end{prop}

\begin{proof}
First we show the map is injective.
Suppose for $\mbl_1, \mbl_2\in {\rm Fun}(\mathbf U, \mathbf V),$ we have,
$${\overline {\mathcal F}}(\mbl_1)={\overline {\mathcal F}}(\mbl_2).$$
That means, for any $\mathbf W\in \Obj(\mathcal C),$ 
\begin{equation}\nonumber
\begin{split}
&\biggl({\overline {\mathcal F}}(\mbl_1)\biggr)(\mathbf W)=\biggl({\overline {\mathcal F}}(\mbl_2)\biggr)(\mathbf W),\\
&\Rightarrow {\overline {\mathcal F}}_{\mathbf W}(\mbl_1)={\overline {\mathcal F}}_{\mathbf W}(\mbl_2),
\end{split}
\end{equation}
and,
by \eqref{D:objfunccov},
\begin{equation}
\begin{split}
{\overline {\mathcal F}}_{\mathbf W}(\mbl_1), {\overline {\mathcal F}}_{\mathbf W}(\mbl_2): &\Obj\biggl({\mathcal F}(\mathbf W, \mathbf U)\biggr)\to \Obj\biggl({\mathcal F}(\mathbf W,  \mathbf V)\biggr),\\
& {\rm Fun}(\mathbf W, \mathbf U)\to {\rm Fun}(\mathbf W, \mathbf V).\end{split}
\end{equation}
Putting $\mathbf W=\mathbf U,$ in the above equation, we obtain, 
$${\overline {\mathcal F}}_{\mathbf U}(\mbl_1), {\overline {\mathcal F}}_{\mathbf U}(\mbl_2): {\rm Fun}(\mathbf U, \mathbf U)\to {\rm Fun}(\mathbf U, \mathbf V).$$
Thus $$\biggl({\overline {\mathcal F}}_{\mathbf U}(\mbl_1)\biggr) ({\mathbf {Id}}_{\mathbf U})=\biggl({\overline {\mathcal F}}_{\mathbf U}(\mbl_2)\biggr) ({\mathbf {Id}}_{\mathbf U}),$$
where ${\mathbf {Id}}_{\mathbf U}$ is the identity functor of $\mathbf U.$ Then the second equation of \eqref{D:objfunccov} implies
$$\mbl_1=\mbl_2.$$

We prove   surjecitivity as follows.

Given any natural transformation \begin{equation}\nonumber
\begin{split}
&{\mathbf \chi}: {\mathcal F}_{\mathbf U}\Longrightarrow \mathcal {\mathcal F}_{\mathbf V},\\
&{\mathbf \chi}(\mathbf W): {\mathcal F}_{\mathbf U}(\mathbf W)\Longrightarrow {\mathcal F}_{\mathbf V}(\mathbf W),\\
&{\mathbf \chi}(\mathbf W): {\mathcal F}({\mathbf W}, \mathbf U)\Longrightarrow {\mathcal F}({\mathbf W}, \mathbf V),
\end{split}
\end{equation}
 we have an element 
$$\mbl:=\biggl({\mathbf \chi}(\mathbf U)\biggr)(\mathbf {Id}_{\mathbf U})\in \Obj\biggl({\mathcal F}({\mathbf U}, \mathbf V)\biggr)={\rm Fun}({\mathbf U}, \mathbf V),$$
which clearly maps to $\mathbf \chi$ under the action of $\overline {\mathcal F}.$

\end{proof}

\begin{exmp}\label{Ex:yon}
\rm{
	Let us consider the case when $\mathcal C={\mathcal O}(\mathbf B)$, where ${\mathcal O}(\mathbf B)$ is defined in \eqref{E:ob}. Thus,  if $\mathbf V$ is a subcategory of $\mathbf U,$ then 	${\mathcal F}(\mathbf V, \mathbf U)$ is an one object category.  Otherwise it is an empty category:
\vskip 0.4 cm 
\begin{equation}
  \left.
  \raisebox{10pt}[30pt]{\smash{$\begin{array}{r@{}l@{\,}l}
   &{\mathcal F}(\mathbf V, \mathbf U)={\mathbf \emptyset}, \qquad \mathbf V\not\subset \mathbf U&\\
   &\Obj\biggl({\mathcal F}(\mathbf V, \mathbf U)\biggr)=\{\mathbf i\}, \mathbf i:\mathbf V\hookrightarrow \mathbf U\\
   &\Mor\biggl({\mathcal F}(\mathbf V, \mathbf U)\biggr)=\{{\mathbf S}\colon\mathbf i\Longrightarrow \mathbf i\} \\
  \end{array}$}}
  \right\} \quad \mathbf V\subset \mathbf U.
\end{equation}
Suppose ${\mathcal F}(\mathbf V, \mathbf U)$ is non empty.
\begin{equation} \nonumber
\xymatrix{
         \ar[d]^{{\mathbf S}(a)} {{\mathbf i}(a)=a}     \ar[rr]^-{{{\mathbf i}(f)}=f} & &{{\mathbf i}(b)=b} \ar[d]_{{\mathbf S}(b)} \\
{\mathbf i}(a)=a \ar[rr]_-{{{\mathbf i}(f)}=f}& & {{\mathbf i}(b)=b}) 
}.
\end{equation}

 Then, it is easy to see that a morphism $\mathbf S\in \Mor\biggl({\mathcal F}(\mathbf V, \mathbf U)\biggr)$ is given by an element ${\mathbf S}(a)$ of ${\rm Hom}(a, a)$ for each $a\in \Obj(\mathbf V),$
$${\mathbf S}=\{{\mathbf S}({a})| \forall a\in \Obj(\mathbf V)\},$$
such that for any $a\xrightarrow{f}b\in \Mor(\mathbf V),$ following condition holds:
\begin{equation}\label{E:almi}
{\mathbf S}(b)\circ f=f\circ{\mathbf S}(a).
\end{equation}
}
Now suppose 	${\mathcal F}(\mathbf V', \mathbf U)={\mathcal F}_{\mathbf U}(\mathbf V')$ and ${\mathcal F}(\mathbf V, \mathbf U)={\mathcal F}_{\mathbf U}(\mathbf V)$ are non empty categories. Also, suppose ${\rm Fun
}(\mathbf V, \mathbf V')$ is non empty. That means we have a filtration of categories:
$$\mathbf V\subset \mathbf V'\subset \mathbf U.$$
Similarly, if ${\mathcal F}(\mathbf V', \mathbf U')={\mathcal F}_{\mathbf U'}(\mathbf V')$ and ${\mathcal F}(\mathbf V, \mathbf U')={\mathcal F}_{\mathbf U'}(\mathbf V)$ are non empty categories, and ${\rm Fun
}(\mathbf V, \mathbf V')$ is non empty, then we have:
$$\mathbf V\subset \mathbf V'\subset \mathbf U'.$$

Now it is immediate that a natural transformation $\chi:{\mathcal F}_{\mathbf U'}\Longrightarrow {\mathcal F}_{\mathbf U}$ exists if and only if
$\mathbf U'\subset \mathbf U.$ Because if following diagram commutes,
\begin{equation} \nonumber
\xymatrix{
	\ar[d]^{{\chi}({\mathbf V'})} {\mathcal F}_{\mathbf U'}(\mathbf V')     \ar[rr]^-{{\mathcal F}_{\mathbf U'}(\mathbf i)} & & {\mathcal F}_{\mathbf U'}(\mathbf V)\ar[d]_{{\chi}(\mathbf V)}\\
{{\mathcal F}_{\mathbf U}(\mathbf V')} \ar[rr]_-{{\mathcal F}_{\mathbf U}(\mathbf i)}& & {\mathcal F}_{\mathbf U}(\mathbf V) 
}
 \end{equation}
then replacing $\mathbf V'$ with $\mathbf  U',$ we will have

\begin{equation} \nonumber
\xymatrix{
	\ar[d]^{{\chi}(\mathbf U')} {\mathcal F}_{\mathbf U'}(\mathbf U')     \ar[rr]^-{{\mathcal F}_{\mathbf U'}(\mathbf i)} & & {\mathcal F}_{\mathbf U'}(\mathbf V)\ar[d]_{{\chi}({\mathbf V})}\\
{{\mathcal F}_{\mathbf U}(\mathbf U')} \ar[rr]_-{{\mathcal F}_{\mathbf U}(\mathbf i)}& & {\mathcal F}_{\mathbf U}(\mathbf V) 
}.
\end{equation}
In order to ${\mathcal F}_{\mathbf U}(\mathbf U')$ be non empty, $\mathbf U'$ must be a subcategory of $\mathbf U:$
$$\mathbf U'\subset \mathbf U.$$   
Suppose $\mathbf U'\subset \mathbf U$  ,  then the  unique natural transformation $\chi$  is given as follows.  For any 
$$\mathbf V'\subset \mathbf U'\subset \mathbf U,$$
and  a given
\begin{eqnarray}
&&{\mathbf S}\in  \Mor\biggl({\mathcal F}_{\mathbf U'}(\mathbf V')\biggr) = \Mor\biggl({\mathcal F}(\mathbf V', \mathbf U')\biggr),\nonumber\\
&&\mathbf S=\Big\{{\mathbf S}(a)\in {\rm Hom}(a, a), \forall a\in \Obj(\mathbf V')| f\circ {\mathbf S}(a)={\mathbf S}(b)\circ f, \forall (a\xrightarrow{f}b)\in {\rm Mor}(\mathbf V')\Big\},\nonumber
\end{eqnarray}
$\biggl(\chi({\mathbf V'})\biggr)(\mathbf S)$ is  simply obtained by treating ${\mathbf S}$ as an element of $ \Mor\biggl({\mathcal F}_{\mathbf U}(\mathbf V')\biggr) ={\rm Mor}\biggl({\mathcal F}(\mathbf V', \mathbf U)\biggr)$ (see the paragraph before \eqref{E:almi}).\fbox{}

\end{exmp}
\section{Sieves of categories}\label{S:cvs}
Recall that the \textit{presheaf of sets}, over a category $\mathbf C,$ is defined as a contravariant functor from $\mathbf C$ to $\mathbf {Set}$
\begin{equation}\label{E:trpresheaf}
	R:{\mathbf C}^{\rm op}\to \mathbf {Set},
\end{equation}
where $\mathbf {Set}$ is the category of (small) sets. For example the contravariant Hom-functor  $${\mathbf {Hom}}(-, c):{\mathbf C}^{\rm op}\to \mathbf {Set},$$
corresponding to each $c\in {\rm Obj}(\mathbf C),$ defines a $\mathbf {Set}$ valued presheaf. Instead of  presheaves of sets, it is possible to introduce additional algebraic structures into the picture. For instance, one may consider presheaves of groups or rings respectively given by  $R:{\mathbf C}^{\rm op}\to \mathbf {Grp}$ or $R:{\mathbf C}^{\rm op}\to \mathbf {Ring}.$ Here $\mathbf {Grp}$ and $\mathbf {Ring}$ respectively denote (locally small) category of (small) groups and (locally small) category of (small) rings. In this section we will focus on the notion of sieves in the context of the category $\mathbf {Cat}.$ 

Let $\mathbf C$ be a category and $c\in {\rm Obj}(\mathbf C).$ A \textit{sieve on} $c$ is a presheaf $R_c:{\mathbf C}^{\rm op}\to \mathbf {Set}$ such that $R_c$ is a sub functor of the contravariant Hom-functor ${\mathbf {Hom}}(-, c);$ that is, for all objects $d\in \Obj(\mathbf C)$ and all morphisms $d\xrightarrow{f}d'\in \Mor(\mathbf C),$ $R_c$ should satisfy
\begin{equation}\label{E:trsieve}
	\begin{split}
		&R_c(d)\subset {\mathbf {Hom}}(d, c),\\
	 &\biggl(R_c(d')\xrightarrow{R_c(f)}R_c(d)\biggr)={\mathbf {Hom}}(f, c)|_{R_c(d')},
\end{split}
\end{equation}
where ${\mathbf {Hom}}(f, c)|_{R_c(d')}$ denotes the restriction of ${{\mathbf {Hom}}(f, c)}:{\mathbf {Hom}}(d', c)\rightarrow{\mathbf {Hom}}(d, c)$ to ${R_c(d')}\subset {\mathbf {Hom}}(d', c).$ 
\subsection{Presheaves of categories}
As before, let $\mathcal C$ be a category of a collection of (small) categories, and $\rmc$ be the category of all small categories. Existence of the functor $${\mathcal F}_{\mathbf U}:{\mathcal C}^{\rm op}\to\rmc,$$ for each $\mathbf U\in \Obj(\mathcal C)$ [Proposition~\ref{Pr:Hompresheaf}], motivates following definition of  \textit{presheaves of categories}.

Let $\mathcal C$ be a category of a collection of (small) categories, and $\rmc$ be the category of all small categories. A   \textit{presheaf of categories} (or, a $\mathbf {Cat}$-\textit{valued presheaf}), over the category $\mathcal C$,  is a functor 
\begin{equation}\label{E:prshfcat}
	{\mathcal R}:{\mathcal C}^{\rm op}\to \rmc.
\end{equation}	
It is immediate from the definition above and Proposition~\ref{Pr:Hompresheaf} that:
\begin{cor}\label{cor:funcprsf}
	For each $\mathbf U\in \Obj(\mathcal C),$ the functor $${\mathcal F}_{\mathbf U}:{\mathcal C}^{\rm op}\to\rmc,$$
	in Proposition~\ref{Pr:Hompresheaf} is a presheaf of categories, over the category $\mathcal C.$
\end{cor}

Let ${{\mathbf {Prsh}}}(\mathcal C, {\mathbf {Cat}}):={\mathcal F}({\mathcal C}^{\rm op}, {\mathbf {Cat}})$ denote the category of $\mathbf {Cat}$-valued presheaves, over the category ${\mathcal C};$ that is,
\begin{equation}\label{E:catprsheaf}
	\begin{split}
		&\Obj\biggl({{\mathbf {Prsh}}}(\mathcal C, {\mathbf {Cat}})\biggr)={\rm Fun}({\mathcal C}^{\rm op}, {\mathbf {Cat}}),\\
  &\Mor\biggl({{\mathbf {Prsh}}}(\mathcal C, {\mathbf {Cat}})\biggr)={\mathcal N}({\mathcal C}^{\rm op}, {\mathbf {Cat}}).
\end{split}
\end{equation}

Then by Corollary~\ref{cor:funcprsf} we have a functor, from the category $\mathcal C$ to the category ${{\mathbf {Prsh}}}(\mathcal C, {\mathbf {Cat}}),$ given by
\begin{eqnarray}
	&&\mathcal C\longrightarrow{{\mathbf {Prsh}}}(\mathcal C, {\mathbf {Cat}}),\nonumber\\
	&\Obj\biggl(\mathcal C\biggr)&\longrightarrow\Obj\biggl({{\mathbf {Prsh}}}(\mathcal C, {\mathbf {Cat}})\biggr)\nonumber\\
 &&{\mathbf U}\mapsto {\mathcal F}_{\mathbf U},\label{E:objemb}\\
 &\Mor\biggl(\mathcal C\biggr)&\longrightarrow\Mor\biggl({{\mathbf {Prsh}}}(\mathcal C, {\mathbf {Cat}})\biggr),\nonumber\\
 &\biggl({\mathbf U}\xrightarrow{\mbl}{\mathbf V}\biggr)&\mapsto \biggl({\overline {\mathcal F}}(\mbl):{\mathcal F}_{\mathbf U}\Longrightarrow {\mathcal F}_{\mathbf V}\biggr), [\hbox{given by} \eqref{E:yonmap}].\label{E:moremb}
\end{eqnarray}

But according to Proposition~\ref{Pr:yoneda}, the above functor  $\mathcal C \longrightarrow {{\mathbf {Prsh}}}(\mathcal C, {\mathbf {Cat}})$ is full and faithful. This allows us to identify  $\mathcal C$  as a full subcategory of
	${{\mathbf {Prsh}}}(\mathcal C, {\mathbf {Cat}}).$
\begin{theorem}\label{Th:yonembed}
	Let $\mathcal C$ be a category of a collection of (small) categories, and $\rmc$ be the category of all small categories. Let ${{\mathbf {Prsh}}}(\mathcal C, {\mathbf {Cat}}):={\mathcal F}({\mathcal C}^{\rm op}, {\mathbf {Cat}})$ be the  category of $\mathbf {Cat}$-valued presheaves over the category ${\mathcal C}.$ Then there exists a full and faithful functor 
	$$\mathcal C\longrightarrow{{\mathbf {Prsh}}}(\mathcal C, {\mathbf {Cat}}),$$
	given by \eqref{E:objemb}--\eqref{E:moremb}. In other words,  $\mathcal C$ can be identified as a full subcategory of
	${{\mathbf {Prsh}}}(\mathcal C, {\mathbf {Cat}}).$
\end{theorem}

Instead of working with the entire category $\mathbf {Cat},$ one can consider a  presheaf of categories with some additional structures. For example, one may define a presheaf of \textit{categorical groups} [3, 10, 14, 18], over $\mathcal C,$ to be a contravariant functor from $\mathcal C$ to $\mathbf {CatGrp}:$ 
$${\mathcal R}: {\mathcal C}^{\rm op}\to \mathbf {CatGrp},$$
where $\mathbf {CatGrp}$ is the category of  categorical groups. We will denote \textit{category of presheaves of categorical groups} by $${{\mathbf {Prsh}}}(\mathcal C, {\mathbf {CatGrp}}).$$ In [13] we  construct such an example of presheaf of categorical groups. In this paper our main objective is to introduce the notion of sieves of categories. So we will work with the definition of presheaf of categories given in \eqref{E:prshfcat}.

Let $\mathcal C$ be a category of a collection of (small) categories, and $\rmc$ be the category of all small categories. Let $\mathbf U\in \Obj(\mathcal C)$. A \textit{sieve of categories} (or, a $\mathbf {Cat}$-\textit{valued sieve}) on $\mathbf U$ is a presheaf of categories
$${\mathcal R}_{\mathbf U}: {\mathcal C}^{op}\to \mathbf {Cat},$$
such that ${\mathcal R}_{\mathbf U}$ is a subfunctor of ${\mathcal F}_{\mathbf U}.$ That is, for any $\mathbf V\in \Obj(\mathcal C),$  and $\biggl({\mathbf V}\xrightarrow{\mbl}{\mathbf V'}\biggr)\in \Mor(\mathcal C),$ we should have
\begin{eqnarray}
&&{\mathcal R}_{\mathbf U}(\mathbf V) \hskip 0.2 cm \hbox {a subcategory of} \hskip 0.2 cm {\mathcal F}(\mathbf V, {\mathbf U});  {\mathcal R}_{\mathbf U}(\mathbf V)\subset {\mathcal F}(\mathbf V, {\mathbf U}),\label{D:sieveobj}\\
&&\biggl({\mathcal R}_{\mathbf U}(\mathbf V ')\xrightarrow{{\mathcal R}_{\mathbf U}(\mbl)}{\mathcal R}_{\mathbf U}(\mathbf V )\biggr)={\mathcal F}_{\mathbf U}(\mbl)|_{{\mathcal R}_{\mathbf U}(\mathbf V')},\label{D:sievemor}
\end{eqnarray}
where ${\mathcal F}_{\mathbf U}(\mbl)|_{{\mathcal R}_{\mathbf U}(\mathbf V')}$ denotes the restriction of the functor ${{\mathcal F}_{\mathbf U}(\mbl)}:{\mathcal F}(\mathbf V', \mathbf U)\rightarrow{\mathcal F}(\mathbf V, \mathbf U)$ to the subcategory ${{\mathcal R}_{\mathbf U}(\mathbf V')}\subset {\mathcal F}(\mathbf V', \mathbf U).$

Equivalently one can think of a ${\mathbf {Cat}}$-valued sieve ${\mathcal R}_{\mathbf U}$ as a collection of morphisms (functors) with common codomain $\mathbf U$ in category $\mathcal C;$ that is, a collection of functors from objects (categories) in $\mathcal C$ to $\mathbf U,$ and a collection of natural transformations between these functors, such that following conditions are satisfied:
\begin{enumerate}
\item{If the functor ${\mathbf V}\xrightarrow{\mathbf \Theta}{\mathbf U}$ is in ${\mathcal R}_{\mathbf U},$ then for any functor ${\mathbf V'}\xrightarrow{\mathbf \Psi}{\mathbf V}\in \Mor(\mathcal C),$   ${\mathbf V'}\xrightarrow{\mathbf \Theta{\mathbf \Psi}}{\mathbf U}$ is also in   ${\mathcal R}_{\mathbf U}.$  }
\item{If the natural transformation ${\mathcal S }\colon \mathbf \Theta_1\Longrightarrow{\mathbf \Theta_2}$ is in ${\mathcal R}_{\mathbf U},$ then for any functor ${\mathbf V'}\xrightarrow{\mathbf \Psi}{\mathbf V}\in \Mor(\mathcal C),$   ${{\mathcal S}\mathbf \Psi}\colon{\mathbf \Theta_1}{\mathbf \Psi}\Longrightarrow{\mathbf \Theta_2}\mathbf \Psi$ is also in   ${\mathcal R}_{\mathbf U}.$}
\end{enumerate}
\begin{exmp}\label{Ex:trivsie}
\rm{
We call a category $\mathbf A $ \textit{trivially discrete}, if objects form a set,  and only morphisms in $\Mor(\mathbf A)$ are identity morphisms:
$$\Mor(\mathbf A)=\{1_{a}| a\in \Obj(\mathbf A)\}\simeq \Obj(\mathbf A).$$
Suppose ${\mathcal C}_{\rm dis}$ is a category, whose objects are  trivially discrete categories. Then a $\mathbf {Cat}$-valued sieve over ${\mathcal C}_{\rm dis}$ is simply a sieve in the traditional sense.\fbox{}}
\end{exmp}

\begin{exmp}\label{Ex:sie}
\rm{
	Let us again consider the case  $\mathcal C={\mathcal O}(\mathbf B).$ We refer to Example~\ref{Ex:yon} for the description of the functors ${\mathcal F}_{\mathbf U}.$ Let $\mathbf U\in \Obj\biggl({\mathcal O}(\mathbf B)\biggr),$  and ${\mathcal R}_{\mathbf U}$ be the \textit{selection} of 
	\begin{enumerate}
	\item{subcategories of $\mathbf U,$ which are stable under  inclusion functors; that is, if ${\mathcal R}_{\mathbf U}$  selects ${\mathbf V}_k, \cdots, {\mathbf V}_1\subset {\mathbf U},$ then
\begin{eqnarray}
	&&\mathbf V_k\subset {\mathbf V}_{k-1}\subset \cdots\cdots \subset{\mathbf V}_2\subset{\mathbf V}_1\subset {\mathbf V}_0=\mathbf U,\nonumber\\
	&&\mathbf V_k\stackrel{{\mathbf i}_k}{\hookrightarrow}{\mathbf V}_{k-1}\stackrel{{\mathbf i}_{k-1}}{\hookrightarrow}\cdots\cdots \stackrel{{\mathbf i}_3}{\hookrightarrow}{\mathbf V}_2\stackrel{{\mathbf i}_2}{\hookrightarrow}{\mathbf V}_1\stackrel{{\mathbf i}_1}{\hookrightarrow}{\mathbf V}_0=\mathbf U,\nonumber
			\end{eqnarray}
}
\item{ a subcategory ${\mathbf C}_{j}$ of  ${\mathcal F}({\mathbf V}_j, {\mathbf U}),$ for each $j, 0\leq j\leq k;$  (we set ${\mathbf i}_0={\mathbf {Id}}_{\mathbf U}$) that is, a choice of  a submonoid ${\mathcal M}_{j}$ of ${\rm Nat}({\mathbf {i}}_0\circ\cdots\circ{\mathbf {i}}_j, {\mathbf {i}}_0\circ\cdots\circ{\mathbf {i}}_j)$ [see Example~\ref{Ex:yon}  for an explicit description of these natural transformations]:
\begin{equation}\nonumber
\begin{split}
&\Obj({\mathbf C}_j)=\{ {\mathbf {i}}_0\circ\cdots\circ{\mathbf {i}}_j \}\\
&\Mor({\mathbf C}_{j})={\mathcal M}_{j}\subset{\rm Nat} ({\mathbf {i}}_0\circ\cdots\circ{\mathbf {i}}_j , {\mathbf {i}}_0\circ\cdots\circ{\mathbf {i}}_j )             ,\qquad 0\leq j\leq k
\end{split}
\end{equation}
such that
$${\mathcal M}_j\subset {{\mathbf i}_{j}^*}{\mathcal M}_{j-1},\qquad 1\leq j\leq k,$$
where ${{\mathbf i}_{j}^*}{\mathcal M}_{j-1}:=\{\mathbf S{\mathbf i}_j|\mathbf S\in {\mathcal M}_{j-1}\}.$

}
\end{enumerate}
Then ${\mathcal R}_{\mathbf U}$ defines a $\mathbf {Cat}$-valued sieve over $\mathbf U\in \Obj\biggl({\mathcal O}(\mathbf B)\biggr)$ as follows. 
\begin{equation}\nonumber
\begin{split}
&{\mathcal R}_{\mathbf U}({\mathbf V}_j)={\mathbf C}_{j},\\
&\biggl({\mathcal R}_{\mathbf U}({\mathbf i}_{s+1}\circ \cdots\circ{\mathbf i}_r)\biggr)(\mathbf S)={\mathbf S}{\mathbf i}_{s+1}\circ \cdots\circ{\mathbf i}_r\in {\rm Mor}({\mathbf C}_r),
\end{split}
\end{equation}
where $r\geq s,$ hence ${{\mathbf i}_{s+1}\circ \cdots\circ{\mathbf i}_r}\colon{\mathbf V}_r{\hookrightarrow}{\mathbf V}_s,$ and ${\mathbf S}\in \Mor({\mathbf C}_s).$

In fact, every $\mathbf {Cat}$-valued sieve over ${\mathcal O}(\mathbf B)$ arises in this fashion.\fbox{}
}
\end{exmp}
\section*{Concluding remarks}
In this paper we have developed a framework of  ${\mathbf {Cat}}$-valued presheaves and sieves  over a category of (small) categories $\mathcal C.$   In [13] we have used some of the constructions and results of this paper to explore ${\mathbf {Cat}}$-valued sheaves over the category ${\widetilde {\mathcal O}}(\mathbf B)$ of ``open subcategories'' of a  topological groupoid  ${\mathbf B}.$

Natural direction of enquiry following ${\mathbf {Cat}}$-valued sieves should be  towards
{Grothendieck topologies} on a topological category. A paper (jointly with, A Lahiri, A N Sengupta) is under preparation on this topic.

Lastly, we are obliged to point out that  ${\mathbf {Cat}}$ is a $2$-category [18], and so is  ${\mathcal C}$ or ${\widetilde {\mathcal O}}(\mathbf B)$.  It would have  done justice to the natural higher structures involved, if we had also incorporated them into our framework and defined the ${\mathbf {Cat}}$-valued presheaf to be a $2$-functor
$${\mathcal C}^{\rm op}\longrightarrow {\mathbf {Cat}}.$$
 For the sake of simplicity, we have treated  ${\mathbf {Cat}}, {\mathcal C}, {\widetilde {\mathcal O}}(\mathbf B)$ as $1$-categories, completely ignoring the higher structures. However, without much difficulty the framework in this paper  extends to higher morphisms as well .

\vskip .3 cm

 {\bf{ Acknowledgments.} } \textit{The author gratefully acknowledges   suggestions received  from Ambar N Sengupta and Amitabha Lahiri. The author would like to thank International Centre for Theoretical Sciences, TIFR, Bangalore for their kind hospitality, where part of this paper was written.}

\end{document}